\documentclass[12pt,oneside]{amsart}
\pdfoutput=1

\usepackage[margin=1in]{geometry}
\usepackage{amsmath,amssymb,amsthm}
\usepackage{graphicx}
\usepackage{tikz}
\usetikzlibrary{arrows.meta,decorations.pathreplacing}
\usepackage[section]{placeins}
\usepackage[colorlinks=true,citecolor=blue,linkcolor=blue,urlcolor=blue]{hyperref}

\numberwithin{equation}{section}
\theoremstyle{plain}
\newtheorem{theorem}{Theorem}[section]
\newtheorem{lemma}[theorem]{Lemma}
\newtheorem{proposition}[theorem]{Proposition}
\newtheorem{corollary}[theorem]{Corollary}
\theoremstyle{definition}
\newtheorem{definition}[theorem]{Definition}
\theoremstyle{remark}
\newtheorem{remark}[theorem]{Remark}

\title[Spectral synthesis and the Fourier ratio]{Spectral synthesis with the complexity parameter}
\author{Shantanu Deodhar}
\address{Department of Mathematics, University of Rochester, Rochester, NY, USA}
\email{sadeodhar99@gmail.com}
\author{Alex Iosevich}
\address{Department of Mathematics, University of Rochester, Rochester, NY, USA}
\email{iosevich@gmail.com}
\subjclass[2020]{Primary 42B10; Secondary 42B20, 28A75}
\keywords{spectral synthesis, Fourier ratio, Fourier extension, Salem measures}
\date{}

\begin{document}

\begin{abstract}
We study how concentration of the Fourier transform changes the integrability threshold imposed by the size of its spectral support. Let $\mu$ be a positive Radon measure on $\mathbb{R}^n$ and let $0\ne f\in L^2(\mu)$ with $f\mu$ compactly supported. Suppose that $\operatorname{supp}(f\mu)$ can be covered by $O(r^{-\alpha})$ balls of radius $r$, where $0<\alpha<n$. If the normalized Fourier ratio of $f\mu$ has decay exponent $\kappa<\frac{\alpha}{2}$, then $\widehat{f\mu}\notin L^p(\mathbb{R}^n)$ for
$$2\leq p<\frac{2(n-2\kappa)}{\alpha-2\kappa},$$
and at $\kappa=\frac{\alpha}{2}$ every finite $p$ is excluded. A uniform power bound for the Fourier ratio excludes the finite critical endpoint as well. We prove that the strict range is sharp for every admissible pair $(\alpha,\kappa)$ by convolving a Salem measure with a Cantor measure whose digits are consecutive and whose bases increase, and the resulting support has Hausdorff dimension exactly $\alpha$. We also obtain necessary conditions for Fourier extension estimates and compute the Fourier ratio exponent for smooth densities on hypersurfaces whose second fundamental form has constant rank.
\end{abstract}

\maketitle

\section{Introduction}

How much does the size of the support of a measure tell us about the integrability of its Fourier transform? Two examples in the plane show that the answer depends on more than size. The Fourier transform of arc length measure on the unit circle decays like $|x|^{-\frac{1}{2}}$ and no faster, so it belongs to $L^p(\mathbb{R}^2)$ precisely when $p>4$. A finite measure on a line segment behaves quite differently, for the absolute value of its Fourier transform is constant in the direction perpendicular to the segment, and a nonzero measure of this kind therefore has no Fourier transform in any finite $L^p$ space. Both supports have dimension one, and it is their geometry rather than their size that accounts for the difference.

Questions of this kind have a long history in the theory of partial differential equations. A tempered distribution $u$ solves the constant coefficient equation $P(D)u=0$ if and only if $\widehat u$ is supported on the zero set of the polynomial $P$, so a theorem asserting that an $L^p$ function whose Fourier transform lives on a given surface must vanish is the statement that the corresponding equation has no nontrivial solutions in $L^p(\mathbb{R}^n)$. The circle corresponds in this way to the Helmholtz equation in the plane. Agranovsky and Narayanan \cite{AN04} proved that if $u\in L^p(\mathbb{R}^n)$ and $\widehat u$ is supported on a $d$-dimensional $C^1$ submanifold, where $0<d<n$, then $u=0$ whenever $1\leq p\leq\frac{2n}{d}$. This belongs to a line of uniqueness results that goes back to the work of Agmon and H\"ormander \cite{AH76} on solutions of equations with simple characteristics, and related questions arise in approximation by spherical waves and in the support theorem for spherical means, for which we refer the reader to \cite{ABK96,AK11}. Following the terminology of this literature, we call vanishing conclusions of this type spectral synthesis results.

Senthil Raani \cite{SR14} extended the dimension bound to fractal supports under a finite packing premeasure hypothesis, and the distinction between packing measure and packing premeasure matters here. Let $N(E,r)$ denote the smallest number of Euclidean balls of radius $r$ needed to cover a bounded set $E\subset\mathbb{R}^n$. As explained in \cite[Remarks~1.1--1.2]{DPR}, the covering condition
\begin{equation}\label{eq:covering}
N(E,r)\leq Cr^{-\alpha}\text{ for }0<r<1,
\end{equation}
is sufficient for the same conclusion, including the endpoint $p=\frac{2n}{\alpha}$. Finite $\alpha$-dimensional packing premeasure implies \eqref{eq:covering}, while finite packing measure alone does not.

Geometry can give a better threshold. Guo \cite{Guo93} determined the sharp range for hypersurfaces with constant relative nullity, Guo, Iosevich, Zhang and Zorin-Kranich \cite{GIZZ23} obtained stronger bounds for the moment curve and its small smooth perturbations, and Duan, Pham and Ryou \cite{DPR} extended the moment curve result to fractal subsets and settled the endpoint. These examples make it natural to ask which properties of a measure, beyond the size of its support, can enter a synthesis estimate.

The quantity we use to measure the concentration of Fourier mass is the Fourier ratio. For a function $f$ on $\mathbb{Z}_N$, the Fourier ratio introduced in \cite{AFR25} is the quotient
$$\frac{\|\widehat f\|_{\ell^1}}{\|\widehat f\|_{\ell^2}},$$
where the norms are taken with respect to counting measure and the discrete Fourier transform is normalized to be unitary. The quotient is small when the Fourier transform is concentrated on a few frequencies and large when the Fourier mass is spread out, which is what makes it a useful measure of the complexity of a signal. Further applications appear in \cite{BIN26}. The continuous version, introduced in \cite{ILPY25}, smooths a compactly supported measure at scale $R^{-1}$ and divides out the factor $R^{\frac{n}{2}}$ that a frequency region of radius $R$ would otherwise contribute.

The reason this ratio enters the synthesis problem is simple. A measure supported on a small set becomes large in $L^2$ when it is smoothed at a fine scale, and on the Fourier side a small $L^1$ norm relative to the $L^2$ norm makes this growth incompatible with an $L^p$ bound over a larger range of $p$. Interpolation joins the two observations. The more substantial question is whether the exponent obtained in this way is sharp for every possible value of the ratio, and the construction in Section~\ref{sec:construction} shows that it is.

\subsection{The Fourier ratio}

We use the Fourier transform
$$\widehat\nu(\xi)=\int e^{-2\pi ix\cdot\xi}d\nu(x)$$
of a finite complex measure $\nu$ on $\mathbb{R}^n$, and the same formula for functions in $L^1(\mathbb{R}^n)$. Positive Radon measures are finite on compact sets throughout the paper, so if $f\in L^2(\mu)$ and $f\mu$ has compact support, then $f\mu$ is a finite complex measure and $\widehat{f\mu}$ is a bounded continuous function. The support $\operatorname{supp}\nu$ of a measure is the smallest closed set carrying all of its mass, $B(x,r)$ is the open Euclidean ball of radius $r$ centered at $x$, and $|A|$ is the Lebesgue measure of a measurable set $A\subset\mathbb{R}^n$.

Fix a nonnegative radial function $\psi\in C_c^\infty(B(0,1))$ with $\int\psi=1$ and put $\psi_\delta(x)=\delta^{-n}\psi(\delta^{-1}x)$ for $\delta>0$. Convolution with $\psi_\delta$ smooths a measure at scale $\delta$, and on the Fourier side it multiplies $\widehat\nu$ by $\widehat{\psi_\delta}(\xi)=\widehat\psi(\delta\xi)$, a Schwartz function equal to one at the origin. For $1\leq p<\infty$ and $R\geq1$, define
\begin{equation}\label{eq:Xdef}
X_{p,\mu,R}(f)
=\left(R^{-n}\int_{\mathbb{R}^n}
|\widehat{f\mu}(\xi)\widehat\psi(R^{-1}\xi)|^pd\xi\right)^{\frac{1}{p}}.
\end{equation}
The factor $R^{-n}$ is the reciprocal of the volume of the frequency ball $B(0,R)$ in which $\widehat\psi(R^{-1}\xi)$ is essentially supported, so $X_{p,\mu,R}(f)$ is an $L^p$ average of $\widehat{f\mu}$ over that ball. For $f\mu\ne0$, the Fourier ratio is
\begin{equation}\label{eq:FRdef}
\operatorname{FR}_{\mu,R}(f)=\frac{X_{1,\mu,R}(f)}{X_{2,\mu,R}(f)}
=R^{-\frac{n}{2}}
\frac{\|\widehat{(f\mu)*\psi_{R^{-1}}}\|_1}
{\|\widehat{(f\mu)*\psi_{R^{-1}}}\|_2}.
\end{equation}
Both norms are finite because $\widehat{f\mu}$ is bounded and $\widehat\psi$ is a Schwartz function, and the denominator is positive because $\widehat\psi$ does not vanish near the origin while the Fourier transform of a nonzero compactly supported measure, being real analytic, cannot vanish on an open set. The factor $R^{-\frac{n}{2}}$ is the reciprocal of the quotient of the $L^1$ and $L^2$ norms of a function of constant modulus on a region of volume $R^n$, so the ratio measures how far $\widehat{f\mu}$ is from being evenly spread over the frequency ball of radius $R$. It is of order one when the Fourier mass fills the ball and small when the mass sits on a thin part of it.

\begin{definition}
For $f\mu\ne0$, define
$$\kappa_\mu(f)=\liminf_{R\to\infty}
\frac{-\log\operatorname{FR}_{\mu,R}(f)}{\log R}.$$
\end{definition}

In other words, $\kappa_\mu(f)$ is the largest exponent $\kappa$ with the property that $\operatorname{FR}_{\mu,R}(f)\leq R^{-\kappa+\varepsilon}$ for every $\varepsilon>0$ and all sufficiently large $R$. We write $A\lesssim B$ when $A\leq CB$ for a constant $C$ independent of the scale under consideration, and $A\gtrsim B$ when $B\lesssim A$. Constants may depend on the fixed measure, function, mollifier and geometric parameters, dependence on a small exponent $\varepsilon$ is indicated by $\lesssim_\varepsilon$, and the covering condition \eqref{eq:covering} will be written $N(E,r)\lesssim r^{-\alpha}$ from now on. If $E=\operatorname{supp}(f\mu)$ satisfies \eqref{eq:covering}, then
$$0\leq\kappa_\mu(f)\leq\frac{\alpha}{2}.$$
We prove this in Lemma~\ref{lem:basic}. The upper bound is the uncertainty estimate from \cite{ILPY25}, and we include its short proof together with a justification of the lower bound for the mollifier used here.

\subsection{The synthesis bound and its sharpness}

\begin{theorem}\label{thm:FR_synthesis}
Let $\mu$ be a positive Radon measure on $\mathbb{R}^n$, let $0\ne f\in L^2(\mu)$, and suppose that $E=\operatorname{supp}(f\mu)$ is compact and satisfies $N(E,r)\lesssim r^{-\alpha}$ for $0<r<1$, where $0<\alpha<n$. Write $\kappa=\kappa_\mu(f)$. If $0\leq\kappa<\frac{\alpha}{2}$, then $\widehat{f\mu}\notin L^p(\mathbb{R}^n)$ for every $p$ with
$$2\leq p<\frac{2(n-2\kappa)}{\alpha-2\kappa}.$$
When $\kappa=0$, the endpoint $p=\frac{2n}{\alpha}$ is also excluded. When $\kappa=\frac{\alpha}{2}$, the Fourier transform $\widehat{f\mu}$ belongs to no $L^p(\mathbb{R}^n)$ with $1\leq p<\infty$.
\end{theorem}

Since $\widehat{f\mu}$ is bounded, membership in $L^p$ for some $p<2$ would imply membership in $L^2$, so the vanishing conclusions extend to $1\leq p<2$ as well. The case $\kappa=0$ recovers the classical dimension threshold $\frac{2n}{\alpha}$, and a positive value of $\kappa$ improves it: the critical exponent $\frac{2(n-2\kappa)}{\alpha-2\kappa}$ increases with $\kappa$ and tends to infinity as $\kappa$ approaches $\frac{\alpha}{2}$, as Figure~\ref{fig:threshold} illustrates. Smooth compactly supported densities on the cylinder $S^1\times\mathbb{R}$ in $\mathbb{R}^3$, for example, have $\alpha=2$ and $\kappa=\frac{1}{2}$, so the threshold becomes $4$ where the dimension bound gives only $3$. The general hypersurface calculation is carried out in Section~\ref{sec:hypersurfaces}.

\begin{figure}[ht]
\centering
\begin{tikzpicture}[font=\small]
\begin{scope}
\fill[blue!10] (0,0) -- (4,0) -- (4,2) -- cycle;
\draw[->] (-0.3,0) -- (4.9,0) node[right] {$\alpha$};
\draw[->] (0,-0.3) -- (0,2.8) node[above] {$\kappa$};
\draw[thick,blue!60!black] (0,0) -- (4,2);
\node[blue!60!black,font=\footnotesize,rotate=26.6] at (1.55,1.15) {$\kappa=\frac{\alpha}{2}$: no finite $p$};
\draw[thick] (0,0) -- (4,0);
\draw[dashed] (4,0) -- (4,2);
\node[below,font=\footnotesize] at (4,0) {$n$};
\node[below left,font=\footnotesize] at (0,0) {$0$};
\node[left,font=\footnotesize] at (0,2) {$\frac{n}{2}$};
\draw[dotted] (0,2) -- (4,2);
\fill (2.8,0.7) circle (1.6pt);
\node[right,font=\footnotesize] at (2.9,0.7) {$(\alpha,\kappa)$};
\node[font=\footnotesize] at (2.2,-0.65) {$\kappa=0$: threshold $\frac{2n}{\alpha}$};
\end{scope}
\begin{scope}[xshift=7.6cm]
\draw[->] (-0.3,0) -- (4.6,0) node[right] {$\kappa$};
\draw[->] (0,-0.3) -- (0,3.4) node[above] {$p$};
\draw[thick,blue!60!black,domain=0:0.85,samples=80,smooth]
  plot ({4*\x},{((3-2*\x)/(1-\x)-2)*0.45});
\draw[dashed] (4,0) -- (4,3.2);
\node[above,font=\footnotesize] at (4,3.2) {$\kappa=\frac{\alpha}{2}$};
\fill (0,0.45) circle (1.6pt);
\node[anchor=north west,font=\footnotesize] at (0.1,0.42) {sphere, $p=3$};
\fill (2,0.9) circle (1.6pt);
\node[anchor=south east,font=\footnotesize] at (1.9,0.98) {cylinder, $p=4$};
\node[below left,font=\footnotesize] at (0,0) {$0$};
\node[below,font=\footnotesize] at (2,0) {$\frac{1}{2}$};
\node[below,font=\footnotesize] at (4,0) {$1$};
\node[left,font=\footnotesize] at (0,0.45) {$3$};
\draw (-0.08,0.45) -- (0.08,0.45);
\node[left,font=\footnotesize] at (0,1.8) {$6$};
\draw (-0.08,1.8) -- (0.08,1.8);
\end{scope}
\end{tikzpicture}
\caption{Left: the admissible parameters in Theorem~\ref{thm:FR_synthesis}. The classical threshold $\frac{2n}{\alpha}$ lives on the segment $\kappa=0$, and on the line $\kappa=\frac{\alpha}{2}$ no finite $p$ is possible. Right: the critical exponent $\frac{2(n-2\kappa)}{\alpha-2\kappa}$ as a function of $\kappa$ for $n=3$ and $\alpha=2$. The sphere and the cylinder of Section~\ref{sec:hypersurfaces} sit at $\kappa=0$ and $\kappa=\frac{1}{2}$.}
\label{fig:threshold}
\end{figure}

The definition of $\kappa$ tolerates losses smaller than every power of $R$, and these losses cannot simply be discarded at the critical exponent, where the interpolation argument has no room to spare. A uniform bound gives the endpoint as well.

\begin{proposition}\label{prop:endpoint}
Under the hypotheses of Theorem~\ref{thm:FR_synthesis}, suppose that for some $0\leq\kappa_0<\frac{\alpha}{2}$,
$$\operatorname{FR}_{\mu,R}(f)\lesssim R^{-\kappa_0}\text{ for }R\geq1.$$
Then $\widehat{f\mu}\notin L^{p_0}(\mathbb{R}^n)$, where
$$p_0=\frac{2(n-2\kappa_0)}{\alpha-2\kappa_0}.$$
\end{proposition}

The strict range in Theorem~\ref{thm:FR_synthesis} is sharp for every admissible pair of parameters. In the statement, $\dim_{\mathcal H}$ denotes Hausdorff dimension.

\begin{theorem}\label{Thm:FractalSharpness}
Let $n\geq1$, $0<\alpha<n$ and $0\leq\kappa<\frac{\alpha}{2}$. There is a compactly supported probability measure $\mu$ on $\mathbb{R}^n$ with the following three properties. Its support $E$ satisfies
$$\dim_{\mathcal H}E=\alpha\text{ and }N(E,r)\lesssim r^{-\alpha}\text{ for }0<r<1.$$
For every $\varepsilon>0$,
$$R^{-\kappa-\varepsilon}\lesssim_\varepsilon
\operatorname{FR}_{\mu,R}(1)
\lesssim_\varepsilon R^{-\kappa+\varepsilon}\text{ for }R\geq1.$$
Finally, $\widehat\mu\in L^p(\mathbb{R}^n)$ for every
$$p>\frac{2(n-2\kappa)}{\alpha-2\kappa}.$$
\end{theorem}

The construction takes place first on the line, where two measures with complementary properties are convolved. One factor is a Salem measure, whose Fourier transform decays pointwise at the rate allowed by its dimension. The other is a Cantor measure built from consecutive digits, whose Fourier transform does not decay pointwise at all but has a small average absolute value over large intervals. Convolution multiplies the Fourier transforms, so the product inherits both estimates, and an $n$-fold product then produces the desired measure in $\mathbb{R}^n$.

\subsection{A necessary condition for extension}

The synthesis theorem also gives an obstruction to an extension estimate, which we take in the form
$$\|\widehat{f\mu}\|_{L^p(\mathbb{R}^n)}
\lesssim\|f\|_{L^q(\mu)}.$$
Its dual is a Fourier restriction estimate. The role of the ratio is to supply a necessary condition that can be combined with a test on a small part of the support.

\begin{corollary}\label{thm:restriction-obstruction}
Let $\mu$ be a compactly supported probability measure on $\mathbb{R}^n$. Suppose that $E=\operatorname{supp}\mu$ satisfies $N(E,r)\lesssim r^{-\alpha}$ for $0<r<1$, where $0<\alpha<n$, and that $\kappa_\mu(1)\geq\kappa$ for some $0\leq\kappa<\frac{\alpha}{2}$. If
$$\|\widehat{f\mu}\|_{L^p(\mathbb{R}^n)}
\lesssim\|f\|_{L^q(\mu)}$$
holds for every $f\in L^q(\mu)$, where $1\leq p<\infty$ and $1\leq q\leq\infty$, then
$$p\geq\max\left\{
\frac{2(n-2\kappa)}{\alpha-2\kappa},
\frac{nq'}{\alpha}
\right\}\text{ and }p>\frac{2n}{\alpha}.$$
Here $q'$ is the conjugate exponent of $q$, with $q'=\infty$ when $q=1$. If, in addition, $\operatorname{FR}_{\mu,R}(1)\lesssim R^{-\kappa}$ for $R\geq1$, then
$$p>\frac{2(n-2\kappa)}{\alpha-2\kappa}.$$
\end{corollary}

\begin{proof}
Taking $f=1$ in the extension estimate shows that $\widehat\mu\in L^p(\mathbb{R}^n)$. If $p<2$, boundedness of $\widehat\mu$ places it in $L^2$ as well, which contradicts Theorem~\ref{thm:FR_synthesis}, so $p\geq2$. The same theorem shows that $\kappa_\mu(1)=\frac{\alpha}{2}$ is impossible, since it would exclude every finite $p$, and that $p$ is at least the critical exponent computed with $\kappa_\mu(1)$ in place of $\kappa$. The critical exponent is an increasing function of $\kappa$, so the hypothesis $\kappa_\mu(1)\geq\kappa$ yields the first of the two lower bounds. The classical endpoint $p=\frac{2n}{\alpha}$ is excluded by the covering hypothesis alone, as the proof of Theorem~\ref{thm:FR_synthesis} shows, and the last assertion of the corollary follows from Proposition~\ref{prop:endpoint}.

The second lower bound comes from testing the estimate on a small ball. Since $\mu(E)=1$ and $E$ can be covered by $O(r^{-\alpha})$ balls of radius $r$, one of these balls, call it $B_r=B(x_r,r)$, satisfies $\mu(B_r)\gtrsim r^\alpha$. Take $f=\chi_{B_r}$, the indicator function of $B_r$. Factoring out the phase $e^{-2\pi ix_r\cdot\xi}$ and taking the real part,
$$|\widehat{\chi_{B_r}\mu}(\xi)|
=\left|\int_{B_r}e^{-2\pi i(x-x_r)\cdot\xi}d\mu(x)\right|
\geq\int_{B_r}\cos(2\pi(x-x_r)\cdot\xi)d\mu(x).$$
When $|\xi|\leq\frac{1}{6}r^{-1}$, the argument of the cosine has absolute value at most $\frac{\pi}{3}$, so the integrand is at least $\frac{1}{2}$ and $|\widehat{\chi_{B_r}\mu}(\xi)|\geq\frac{1}{2}\mu(B_r)$ on a ball of volume comparable to $r^{-n}$. The extension estimate therefore implies
$$r^{-\frac{n}{p}}\mu(B_r)
\lesssim\|\widehat{\chi_{B_r}\mu}\|_{L^p(\mathbb{R}^n)}
\lesssim\|\chi_{B_r}\|_{L^q(\mu)}=\mu(B_r)^{\frac{1}{q}}.$$
Dividing by $\mu(B_r)^{\frac{1}{q}}$ and using the lower bound on $\mu(B_r)$, we obtain
$$1\gtrsim r^{-\frac{n}{p}}\mu(B_r)^{\frac{1}{q'}}
\gtrsim r^{\frac{\alpha}{q'}-\frac{n}{p}},$$
and letting $r$ tend to zero forces $\frac{n}{p}\leq\frac{\alpha}{q'}$. When $q=1$ the right side is zero and no finite $p$ survives.
\end{proof}

\section{Proof of the synthesis estimates}

We first record the estimates that relate the covering hypothesis to the Fourier ratio. They also explain why smoothing is the right operation in this problem: a measure on a thin set has no $L^2$ norm of its own, but its smoothings at scale $R^{-1}$ do, and their $L^2$ norms grow at a rate governed by the covering number.

\begin{lemma}\label{lem:basic}
Let $\nu$ be a nonzero finite complex measure with compact support $E\subset\mathbb{R}^n$, and suppose that $N(E,r)\lesssim r^{-\alpha}$ for $0<r<1$, where $0<\alpha<n$. Put
$$g_R=\nu*\psi_{R^{-1}}\text{ and }H_R=\widehat{g_R},$$
so that $H_R(\xi)=\widehat\nu(\xi)\widehat\psi(R^{-1}\xi)$. Then
\begin{equation}\label{eq:L2lower}
\|H_R\|_2\gtrsim R^{\frac{n-\alpha}{2}}
\end{equation}
for all sufficiently large $R$, and for every $\varepsilon>0$,
\begin{equation}\label{eq:basic-ratio}
R^{-\frac{\alpha}{2}}
\lesssim R^{-\frac{n}{2}}\frac{\|H_R\|_1}{\|H_R\|_2}
\lesssim_\varepsilon R^\varepsilon.
\end{equation}
\end{lemma}

When $\nu=f\mu$, the middle term in \eqref{eq:basic-ratio} is $\operatorname{FR}_{\mu,R}(f)$ by \eqref{eq:FRdef}, and taking logarithms, dividing by $\log R$ and letting $R\to\infty$ in \eqref{eq:basic-ratio} gives the two-sided bound $0\leq\kappa_\mu(f)\leq\frac{\alpha}{2}$ stated in the introduction.

\begin{proof}
Let $E_r=\{x\in\mathbb{R}^n:\operatorname{dist}(x,E)<r\}$ be the open $r$-neighborhood of $E$. Doubling the radii in a cover of $E$ by $N(E,r)$ balls of radius $r$ produces a cover of $E_r$ by the same number of balls of radius $2r$, so the covering hypothesis gives
\begin{equation}\label{eq:tube}
|E_r|\lesssim r^{n-\alpha}.
\end{equation}
Choose $\varphi\in C_c^\infty(\mathbb{R}^n)$ with $\langle\nu,\varphi\rangle=\int\varphi d\nu\ne0$, which is possible because $\nu\ne0$. Since $g_R$ converges to $\nu$ in the sense of distributions as $R\to\infty$, the pairing $|\langle g_R,\varphi\rangle|$ is bounded below by a positive constant for all large $R$. On the other hand, $g_R$ is supported in $E_{R^{-1}}$ because $\psi_{R^{-1}}$ is supported in $B(0,R^{-1})$, and the Cauchy--Schwarz inequality gives
$$|\langle g_R,\varphi\rangle|
\leq\|\varphi\|_\infty\int_{E_{R^{-1}}}|g_R(x)|dx
\leq\|\varphi\|_\infty|E_{R^{-1}}|^{\frac{1}{2}}\|g_R\|_2.$$
By Plancherel's theorem $\|g_R\|_2=\|H_R\|_2$, and the inequality \eqref{eq:tube} with $r=R^{-1}$ turns the last display into \eqref{eq:L2lower}.

The lower bound in \eqref{eq:basic-ratio} comes from the same volume estimate. The function $g_R$ is continuous with integrable Fourier transform, so Fourier inversion gives $\|g_R\|_\infty\leq\|H_R\|_1$, and therefore
$$\|H_R\|_2=\|g_R\|_2
\leq|E_{R^{-1}}|^{\frac{1}{2}}\|g_R\|_\infty
\leq|E_{R^{-1}}|^{\frac{1}{2}}\|H_R\|_1.$$
Dividing by $R^{\frac{n}{2}}\|H_R\|_2$ and using \eqref{eq:tube} once more, we obtain
$$R^{-\frac{n}{2}}\frac{\|H_R\|_1}{\|H_R\|_2}
\gtrsim R^{-\frac{n}{2}}R^{\frac{n-\alpha}{2}}=R^{-\frac{\alpha}{2}}.$$

For the upper bound, fix $\delta>0$ and split the $L^1$ norm of $H_R$ at radius $R^{1+\delta}$. On the ball $|\xi|\leq R^{1+\delta}$, the Cauchy--Schwarz inequality gives
$$\int_{|\xi|\leq R^{1+\delta}}|H_R(\xi)|d\xi
\lesssim R^{\frac{n(1+\delta)}{2}}\|H_R\|_2.$$
Outside this ball we use the boundedness of $\widehat\nu$ and the rapid decay of $\widehat\psi$, which give $|H_R(\xi)|\lesssim_N(R^{-1}|\xi|)^{-N}$ for every $N>n$, so that
$$\int_{|\xi|>R^{1+\delta}}|H_R(\xi)|d\xi
\lesssim_NR^N\int_{|\xi|>R^{1+\delta}}|\xi|^{-N}d\xi
\lesssim_NR^{n+\delta(n-N)}.$$
After division by $R^{\frac{n}{2}}\|H_R\|_2$, the first term contributes $R^{\frac{n\delta}{2}}$, and by \eqref{eq:L2lower} the second contributes at most a constant times $R^{\frac{n}{2}+\delta(n-N)-\frac{n-\alpha}{2}}=R^{\frac{\alpha}{2}+\delta(n-N)}$. Given $\varepsilon>0$, choose $\delta$ with $\frac{n\delta}{2}\leq\varepsilon$ and then $N$ so large that $\frac{\alpha}{2}+\delta(n-N)\leq0$. Both contributions are then bounded by a constant times $R^\varepsilon$, which proves the upper bound in \eqref{eq:basic-ratio}.
\end{proof}

\begin{proof}[Proof of Theorem~\ref{thm:FR_synthesis}]
Write $\nu=f\mu$ and keep the notation of Lemma~\ref{lem:basic}, so that $H_R=\widehat\nu\widehat\psi(R^{-1}\cdot)$. Suppose, for a contradiction, that $\widehat\nu\in L^p(\mathbb{R}^n)$ for some $2\leq p<\infty$, and fix $\varepsilon>0$. By the definition of $\kappa$ as a lower limit, $\operatorname{FR}_{\mu,R}(f)\leq R^{-\kappa+\varepsilon}$ for all sufficiently large $R$, which in the notation of the lemma reads
$$\|H_R\|_1\leq R^{\frac{n}{2}-\kappa+\varepsilon}\|H_R\|_2.$$
We now bring in the $L^p$ norm. H\"older's inequality interpolates the $L^2$ norm between the $L^1$ and $L^p$ norms,
$$\|H_R\|_2\leq\|H_R\|_1^\theta\|H_R\|_p^{1-\theta}
\text{ with }\theta=\frac{p-2}{2(p-1)},$$
where the value of $\theta$ is dictated by $\frac{1}{2}=\theta+\frac{1-\theta}{p}$. Inserting the ratio bound, solving for $\|H_R\|_2$, and using $\frac{\theta}{1-\theta}=\frac{p-2}{p}$ together with $\|H_R\|_p\leq\|\widehat\psi\|_\infty\|\widehat\nu\|_p$, we find that
$$\|H_R\|_2
\lesssim R^{(\frac{n}{2}-\kappa+\varepsilon)\frac{p-2}{p}}
\|\widehat\nu\|_p.$$
This upper bound competes with the lower bound \eqref{eq:L2lower}, and the comparison gives
\begin{equation}\label{eq:contradiction}
1\lesssim
R^{\frac{p(\alpha-2\kappa)-2(n-2\kappa)}{2p}
+\varepsilon\frac{p-2}{p}}\|\widehat\nu\|_p.
\end{equation}
If $\kappa<\frac{\alpha}{2}$ and $p<\frac{2(n-2\kappa)}{\alpha-2\kappa}$, the first exponent is negative, and it stays negative after the term $\varepsilon\frac{p-2}{p}$ is added as long as $\varepsilon$ is small enough, and letting $R\to\infty$ in \eqref{eq:contradiction} then produces a contradiction. If $\kappa=\frac{\alpha}{2}$, the exponent is
$$\frac{\alpha-n}{p}+\varepsilon\frac{p-2}{p},$$
which is again negative for small $\varepsilon$ because $\alpha<n$, and no finite $p\geq2$ survives. Together with the observation, made after the statement of the theorem, that boundedness of $\widehat\nu$ rules out $1\leq p<2$, this proves the assertion at maximal concentration.

It remains to see that the covering hypothesis by itself excludes the classical endpoint $p_*=\frac{2n}{\alpha}$. Suppose that $\widehat\nu\in L^{p_*}(\mathbb{R}^n)$ and fix $M>0$. The contribution of the ball $|\xi|\leq M$ to $\|H_R\|_2$ is at most $\|\widehat\nu\|_\infty\|\widehat\psi\|_\infty|B(0,M)|^{\frac{1}{2}}$, a constant $C_M$ independent of $R$. On the complementary region we apply H\"older's inequality with the exponents $\frac{p_*}{2}$ and $\frac{p_*}{p_*-2}$ and take square roots, which gives
$$\|H_R\|_{L^2(|\xi|>M)}
\leq\|\widehat\nu\|_{L^{p_*}(|\xi|>M)}
\|\widehat\psi(R^{-1}\cdot)\|_{L^{\frac{2p_*}{p_*-2}}(\mathbb{R}^n)}.$$
A change of variables shows that the last norm equals $R^{\frac{n(p_*-2)}{2p_*}}$ times a constant, and the choice of $p_*$ makes the exponent equal to $\frac{n}{2}-\frac{n}{p_*}=\frac{n-\alpha}{2}$. Altogether,
$$\|H_R\|_2
\leq C_M+CR^{\frac{n-\alpha}{2}}
\|\widehat\nu\|_{L^{p_*}(|\xi|>M)}.$$
Divide by $R^{\frac{n-\alpha}{2}}$ and let $R\to\infty$. By \eqref{eq:L2lower} the left side stays bounded below by a positive constant, while the first term on the right disappears, so the tail $\|\widehat\nu\|_{L^{p_*}(|\xi|>M)}$ is bounded below by a positive constant independent of $M$. Since this tail tends to zero as $M\to\infty$, we have reached a contradiction. Finally, as noted earlier, boundedness of $\widehat\nu$ excludes $1\leq p<2$ as well.
\end{proof}

\begin{proof}[Proof of Proposition~\ref{prop:endpoint}]
Suppose that $\widehat\nu\in L^{p_0}(\mathbb{R}^n)$, where $\nu=f\mu$, and write $A_R=\|H_R\|_2$, $a=\frac{n}{2}-\kappa_0$ and $\theta=\frac{p_0-2}{2(p_0-1)}$, so that $\frac{\theta}{1-\theta}=\frac{p_0-2}{p_0}$ and, by the choice of $p_0$, $a\frac{p_0-2}{p_0}=\frac{n-\alpha}{2}$. The assumed ratio bound says that $\|H_R\|_1\leq CR^aA_R$ for all $R\geq1$, with no loss of a power $R^\varepsilon$, and this is what makes the endpoint accessible. Fix $M>0$ and put
$$\eta_M=\|\widehat\nu\|_{L^{p_0}(|\xi|>M)},$$
which tends to zero as $M\to\infty$. As in the previous proof, the contribution of the ball $|\xi|\leq M$ to $A_R$ is bounded by a constant $C_M$ independent of $R$, and on the complementary region H\"older's inequality gives
$$\|H_R\|_{L^2(|\xi|>M)}
\leq\|H_R\|_1^\theta\|H_R\|_{L^{p_0}(|\xi|>M)}^{1-\theta}
\leq C(R^aA_R)^\theta\eta_M^{1-\theta}.$$
Thus $A_R\leq C_M+C(R^aA_R)^\theta\eta_M^{1-\theta}$. If $A_R\leq2C_M$ there is nothing to prove. Otherwise $A_R-C_M\geq\frac{1}{2}A_R$, and dividing by $A_R^\theta$ and raising to the power $\frac{1}{1-\theta}$ gives $A_R\leq CR^{\frac{a\theta}{1-\theta}}\eta_M$. In either case,
$$A_R\leq2C_M+CR^{\frac{n-\alpha}{2}}\eta_M,$$
with $C$ independent of $M$ and $R$. Dividing by $R^{\frac{n-\alpha}{2}}$, letting $R\to\infty$ and then letting $M\to\infty$ contradicts \eqref{eq:L2lower} exactly as before.
\end{proof}

\section{Construction of sharp examples}\label{sec:construction}

The construction rests on two estimates of different kinds. A Salem measure of dimension $s$ on the line has pointwise Fourier decay of order $\frac{s}{2}$, up to an arbitrarily small loss, and this is the most decay its dimension allows. For the structured measure $\rho_t$ of dimension $t$ we want instead an average estimate of the form
$$\int_{-R}^R|\widehat\rho_t(\xi)|d\xi
\lesssim_\varepsilon R^{1-t+\varepsilon},$$
which says that $|\widehat\rho_t|$ behaves like $R^{-t}$ on average over the interval $[-R,R]$ even though it returns to the value one infinitely often. After convolution the Fourier transforms multiply, and so do the two estimates. The choices of $s$ and $t$ are then dictated by $\alpha$ and $\kappa$.

\subsection{The structured measure}

The following construction keeps consecutive digits at each stage. The set and its measure are deterministic, but independent random variables provide a convenient description of the measure and make the Fourier transform factor.

\begin{lemma}\label{lem:arithmetic}
For every $0\leq t<1$, there is a compactly supported probability measure $\rho_t$ on $\mathbb{R}$, with support $A_t$, such that $\dim_{\mathcal H}A_t=t$, such that $N(A_t,r)\lesssim_tr^{-t}$ and $\rho_t(B(x,r))\gtrsim_tr^t$ for $x\in A_t$ and $0<r<1$, and such that for every $\varepsilon>0$,
\begin{equation}\label{eq:arithmetic-L1}
\int_{-R}^R|\widehat\rho_t(\xi)|d\xi
\lesssim_{t,\varepsilon}R^{1-t+\varepsilon}\text{ for }R\geq1.
\end{equation}
\end{lemma}

\begin{proof}
For $t=0$ we take $\rho_0=\delta_0$, the unit mass at the origin, for which every assertion is immediate. Suppose that $0<t<1$. Choose an integer $J$ so large that $(J+1)^t\geq6$, and set $n_j=J+j$, $L_j=\left\lfloor\frac{n_j^t}{2}\right\rfloor$, $B_j=\prod_{i=1}^jn_i$ and $T_j=\prod_{i=1}^jL_i$, with $B_0=T_0=1$. The choice of $J$ guarantees that $3\leq L_j<n_j$. The set $A_t$ is built in levels. At level $j$, each interval from the preceding level is divided into $n_j$ equal parts, of which the first $L_j$ are kept and the rest discarded. There are $T_j$ retained intervals at level $j$, each of length $B_j^{-1}$, and $A_t$ is the intersection over $j$ of the unions of the retained intervals. Figure~\ref{fig:cantor} shows the first two levels.

\begin{figure}[ht]
\centering
\resizebox{\linewidth}{!}{
\begin{tikzpicture}[x=0.7cm,y=0.8cm,>=Stealth,line cap=round]
\tikzset{
  keep0/.style={fill=gray!22,draw=gray!70!black,thick},
  keep1/.style={fill=blue!22,draw=blue!55!black,thick},
  keep2/.style={fill=violet!28,draw=violet!60!black,thick},
  drop/.style={draw=gray!55,dashed,fill=none},
  pick/.style={fill=orange!45,draw=orange!70!black,very thick}
}

\draw[keep0] (0,0) rectangle (12,0.6);
\node[anchor=east,font=\small] at (-0.3,0.3) {level $0$};
\node[anchor=west,font=\small] at (12.9,0.3) {$T_0=1$, length $1$};

\foreach \i in {0,...,7}{
  \ifnum\i<3
    \draw[keep1] ({\i*1.5},-2) rectangle ({(\i+1)*1.5},-1.4);
  \else
    \draw[drop]  ({\i*1.5},-2) rectangle ({(\i+1)*1.5},-1.4);
  \fi
}
\draw[pick] (1.5,-2) rectangle (3,-1.4);
\node[anchor=east,font=\small] at (-0.3,-1.7) {level $1$};
\node[anchor=west,font=\small] at (12.9,-1.7)
  {$T_1=L_1$, length $B_1^{-1}=\frac{1}{n_1}$};
\draw[decorate,decoration={brace,amplitude=4pt}] (0,-1.25) -- (4.5,-1.25)
  node[midway,above=4pt,font=\small] {keep the first $L_1$ of $n_1$ subintervals};
\draw[dotted,gray!70] (0,0) -- (0,-1.4) (4.5,0) -- (4.5,-1.4);

\draw[dotted,gray!90] (1.5,-2) -- (0,-3.4) (3,-2) -- (12,-3.4);
\foreach \m in {0,...,8}{
  \ifnum\m<3
    \draw[keep2] ({\m*1.3333333},-4) rectangle ({(\m+1)*1.3333333},-3.4);
  \else
    \draw[drop]  ({\m*1.3333333},-4) rectangle ({(\m+1)*1.3333333},-3.4);
  \fi
}
\draw[pick] (2.6666667,-4) rectangle (4,-3.4);
\node[anchor=east,font=\small] at (-0.3,-3.7) {level $2$};
\node[anchor=west,font=\small,align=left] at (12.9,-3.7)
  {$T_2=L_1L_2$, length $B_2^{-1}=\frac{1}{n_1n_2}$\\
   {\footnotesize (zoom on one level-1 interval)}};
\draw[decorate,decoration={brace,amplitude=3pt}] (0,-3.3) -- (4,-3.3);
\node[anchor=west,font=\footnotesize] at (4.3,-3.0) {keep the first $L_2$ of $n_2$};

\fill[orange!70!black] (2.6666667,-4) circle (2pt);
\draw[->,orange!70!black,thin] (2.6666667,-4.95) -- (2.6666667,-4.1);
\node[orange!70!black,font=\small,anchor=north] at (2.6666667,-4.95)
  {$S_2=\frac{X_1}{B_1}+\frac{X_2}{B_2}$};
\node[orange!70!black,font=\footnotesize,anchor=north west] at (3.5,-4.15) {$X_2=2$};
\node[orange!70!black,font=\footnotesize,anchor=north] at (2.25,-2.25) {$X_1=1$};

\node[font=\Large] at (6,-6.3) {$\vdots$};
\end{tikzpicture}}
\caption{The first two levels of the Cantor construction with $t=0.9$ and $J=7$, so that $n_1=8$, $n_2=9$ and $L_1=L_2=3$. The highlighted interval records one choice of the digits.}
\label{fig:cantor}
\end{figure}

Let $X_1,X_2,\ldots$ be independent random variables such that $X_j$ is uniformly distributed on the set of digits $\{0,\ldots,L_j-1\}$, and define
$$S_j=\sum_{i=1}^j\frac{X_i}{B_i}\text{ and }X=\sum_{i=1}^\infty\frac{X_i}{B_i}.$$
Since $X_i\leq n_i-1$, the tail of the series after the $j$th term is at most
$$\sum_{i>j}\frac{n_i-1}{B_i}
=\sum_{i>j}\left(\frac{1}{B_{i-1}}-\frac{1}{B_i}\right)
=\frac{1}{B_j},$$
so the series converges, $X$ takes values in $[0,1]$, and $X$ lies in the retained level-$j$ interval $[S_j,S_j+B_j^{-1}]$ determined by its first $j$ digits. Let $\mu_j$ be the distribution of $S_j$, which assigns mass $T_j^{-1}$ to each of the $T_j$ left endpoints of the retained level-$j$ intervals, and let $\rho_t$ be the distribution of $X$. Each retained level-$j$ interval carries $\rho_t$-mass exactly $T_j^{-1}$, and the support of $\rho_t$ is $A_t$.

For $L\geq1$ write
$$D_L(u)=\frac{1}{L}\sum_{a=0}^{L-1}e^{-2\pi iau},$$
the Fourier transform of the uniform distribution on $\{0,\ldots,L-1\}$. Independence of the digits gives
$$\widehat\mu_j(\xi)
=\prod_{i=1}^jD_{L_i}(B_i^{-1}\xi)
=\widehat\mu_{j-1}(\xi)D_{L_j}(B_j^{-1}\xi).$$
Each factor $D_{L_i}(B_i^{-1}\xi)$ is $B_i$-periodic and $B_i$ divides $B_{j-1}$ when $i\leq j-1$, so $\widehat\mu_{j-1}$ is $B_{j-1}$-periodic. We split $[0,B_j]$ into the $n_j$ intervals $[aB_{j-1},(a+1)B_{j-1}]$, substitute $\xi=aB_{j-1}+u$ in each of them, and use the periodicity of $\widehat\mu_{j-1}$ together with $B_j=n_jB_{j-1}$, which gives
$$\int_0^{B_j}|\widehat\mu_j(\xi)|d\xi
=\int_0^{B_{j-1}}|\widehat\mu_{j-1}(u)|
\sum_{a=0}^{n_j-1}
\left|D_{L_j}\left(\frac{a+B_{j-1}^{-1}u}{n_j}\right)\right|du.$$
The inner sum samples the Dirichlet-type kernel $|D_{L_j}|$ at $n_j$ equally spaced points of a period. The kernel equals one at the integers, but the geometric sum bound
$$|D_L(v)|\leq\min\left\{1,\frac{1}{2L\|v\|}\right\},$$
where $\|v\|$ is the distance from $v$ to the nearest integer, shows that it is small on most of the period, and summing this bound over the equally spaced points gives
$$\sup_{v\in\mathbb{R}}\sum_{a=0}^{n_j-1}
\left|D_{L_j}\left(\frac{a+v}{n_j}\right)\right|
\lesssim\frac{n_j}{L_j}\log(2L_j)$$
with an absolute constant. This is \cite[Lemma~5.1]{CVY} once the exponential sum there is divided by $L_j$, and Figure~\ref{fig:dirichlet} illustrates the mechanism. Inserting this bound into the identity above and iterating $j$ times, with $\mu_0=\delta_0$ and $\int_0^{B_0}|\widehat\mu_0|=1$ at the last step, and then using the $B_j$-periodicity and the evenness of $|\widehat\mu_j|$ to pass from $[0,B_j]$ to $[-B_j,B_j]$, we arrive at
\begin{equation}\label{eq:finite-L1}
\int_{-B_j}^{B_j}|\widehat\mu_j(\xi)|d\xi
\leq\frac{2B_j}{T_j}\prod_{i=1}^jC\log(2L_i).
\end{equation}

\begin{figure}[ht]
\centering
\begin{tikzpicture}[x=10cm,y=3cm,font=\small]
\draw[->] (-0.53,0) -- (0.56,0) node[right] {$u$};
\draw[->] (0,-0.05) -- (0,1.12) node[above] {$|D_L(u)|$};
\draw[thick,blue!60!black] plot coordinates {
(-0.5000,0.000) (-0.4938,0.020) (-0.4875,0.039) (-0.4813,0.057) (-0.4750,0.074) (-0.4688,0.089) (-0.4625,0.102)
(-0.4562,0.112) (-0.4500,0.120) (-0.4437,0.125) (-0.4375,0.127) (-0.4313,0.126) (-0.4250,0.122) (-0.4188,0.115)
(-0.4125,0.105) (-0.4062,0.092) (-0.4000,0.077) (-0.3937,0.060) (-0.3875,0.041) (-0.3812,0.021) (-0.3750,0.000)
(-0.3688,0.021) (-0.3625,0.043) (-0.3563,0.063) (-0.3500,0.082) (-0.3438,0.100) (-0.3375,0.116) (-0.3312,0.129)
(-0.3250,0.139) (-0.3187,0.147) (-0.3125,0.150) (-0.3063,0.150) (-0.3000,0.147) (-0.2938,0.140) (-0.2875,0.129)
(-0.2812,0.114) (-0.2750,0.097) (-0.2687,0.076) (-0.2625,0.053) (-0.2562,0.027) (-0.2500,0.000) (-0.2438,0.028)
(-0.2375,0.057) (-0.2313,0.085) (-0.2250,0.113) (-0.2188,0.139) (-0.2125,0.163) (-0.2062,0.185) (-0.2000,0.202)
(-0.1937,0.216) (-0.1875,0.225) (-0.1813,0.229) (-0.1750,0.228) (-0.1688,0.220) (-0.1625,0.207) (-0.1562,0.188)
(-0.1500,0.162) (-0.1437,0.130) (-0.1375,0.092) (-0.1312,0.049) (-0.1250,0.000) (-0.1188,0.054) (-0.1125,0.112)
(-0.1063,0.173) (-0.1000,0.238) (-0.0938,0.304) (-0.0875,0.373) (-0.0812,0.441) (-0.0750,0.509) (-0.0687,0.576)
(-0.0625,0.641) (-0.0563,0.702) (-0.0500,0.760) (-0.0438,0.813) (-0.0375,0.860) (-0.0312,0.902) (-0.0250,0.936)
(-0.0187,0.964) (-0.0125,0.984) (-0.0062,0.996) (0.0000,1.000) (0.0062,0.996) (0.0125,0.984) (0.0188,0.964)
(0.0250,0.936) (0.0312,0.902) (0.0375,0.860) (0.0437,0.813) (0.0500,0.760) (0.0563,0.702) (0.0625,0.641)
(0.0687,0.576) (0.0750,0.509) (0.0813,0.441) (0.0875,0.373) (0.0938,0.304) (0.1000,0.238) (0.1062,0.173)
(0.1125,0.112) (0.1188,0.054) (0.1250,0.000) (0.1312,0.049) (0.1375,0.092) (0.1438,0.130) (0.1500,0.162)
(0.1562,0.188) (0.1625,0.207) (0.1687,0.220) (0.1750,0.228) (0.1813,0.229) (0.1875,0.225) (0.1937,0.216)
(0.2000,0.202) (0.2063,0.185) (0.2125,0.163) (0.2188,0.139) (0.2250,0.113) (0.2312,0.085) (0.2375,0.057)
(0.2438,0.028) (0.2500,0.000) (0.2562,0.027) (0.2625,0.053) (0.2688,0.076) (0.2750,0.097) (0.2812,0.114)
(0.2875,0.129) (0.2937,0.140) (0.3000,0.147) (0.3063,0.150) (0.3125,0.150) (0.3187,0.147) (0.3250,0.139)
(0.3313,0.129) (0.3375,0.116) (0.3438,0.100) (0.3500,0.082) (0.3562,0.063) (0.3625,0.043) (0.3688,0.021)
(0.3750,0.000) (0.3812,0.021) (0.3875,0.041) (0.3938,0.060) (0.4000,0.077) (0.4062,0.092) (0.4125,0.105)
(0.4187,0.115) (0.4250,0.122) (0.4313,0.126) (0.4375,0.127) (0.4437,0.125) (0.4500,0.120) (0.4563,0.112)
(0.4625,0.102) (0.4688,0.089) (0.4750,0.074) (0.4812,0.057) (0.4875,0.039) (0.4938,0.020) (0.5000,0.000)
};
\draw[dashed,red!70!black,domain=0.0625:0.5,samples=60] plot (\x,{1/(16*\x)});
\draw[dashed,red!70!black,domain=-0.5:-0.0625,samples=60] plot (\x,{-1/(16*\x)});
\node[red!70!black,anchor=south west] at (0.27,0.27) {$\frac{1}{2L\|u\|}$};
\foreach \a in {0,...,39}{
  \fill[black] ({-0.5+(\a+0.3)/40},-0.025) circle (1pt);
}
\foreach \pos/\lab in {-0.5/$-\frac{1}{2}$,-0.125/$-\frac{1}{L}$,0.125/$\frac{1}{L}$,0.5/$\frac{1}{2}$}{
  \draw (\pos,-0.015) -- (\pos,0.015);
  \node[below=4pt] at (\pos,-0.03) {\lab};
}
\end{tikzpicture}
\caption{The kernel $|D_L(u)|$ for $L=8$ over one period, with the envelope $\frac{1}{2L\|u\|}$. The dots mark $n=40$ equally spaced sample points $\frac{a+v}{n}$. Only the samples that fall into the main lobe, which has width about $\frac{2}{L}$, see a value of order one, and the remaining samples contribute a sum of order $\frac{n}{L}\log L$ through the envelope.}
\label{fig:dirichlet}
\end{figure}

The digit counts satisfy
\begin{equation}\label{eq:digit-count}
3^{-j}B_j^t\leq T_j\leq2^{-j}B_j^t,
\end{equation}
because the choice of $J$ makes $n_i^t\geq6$ and hence $\frac{n_i^t}{3}\leq L_i\leq\frac{n_i^t}{2}$. Since $\log B_j=\sum_{i\leq j}\log(J+i)$ has order $j\log j$, the inequality \eqref{eq:digit-count} gives $\log T_j=t\log B_j+O(j)$, while the sum $\sum_{i\leq j}\log(C\log(2L_i))$ is at most $j\log(C\log(2n_j))$, which is $O(j\log\log j)$ and hence $o(\log B_j)$. The right side of \eqref{eq:finite-L1} is therefore at most $C_\varepsilon B_j^{1-t+\varepsilon}$ for every $\varepsilon>0$.

To pass from $\mu_j$ to $\rho_t$, observe that $X-S_j$ is independent of $S_j$, so $\rho_t$ is the convolution of $\mu_j$ with the distribution of $X-S_j$, and the Fourier transform of the latter has modulus at most one. Thus $|\widehat\rho_t(\xi)|\leq|\widehat\mu_j(\xi)|$ for every $j$. Given $R\geq1$, choose $j$ with $B_{j-1}\leq R<B_j$. Then
$$\int_{-R}^R|\widehat\rho_t(\xi)|d\xi
\leq\int_{-B_j}^{B_j}|\widehat\mu_j(\xi)|d\xi
\lesssim_\varepsilon B_j^{1-t+\frac{\varepsilon}{2}}
\lesssim_\varepsilon R^{1-t+\varepsilon},$$
where the last step uses $B_j=n_jB_{j-1}\leq n_jR$ and the fact that $\log n_j=o(\log B_{j-1})$, that is, $\frac{\log B_j}{\log B_{j-1}}\to1$. This proves \eqref{eq:arithmetic-L1}.

We now check the geometric assertions at all scales. Let $B_j^{-1}\leq r<B_{j-1}^{-1}$. The $T_j$ retained intervals of level $j$ have length $B_j^{-1}\leq r$ and cover $A_t$, so the inequality \eqref{eq:digit-count} gives
$$N(A_t,r)\leq T_j
\leq2^{-j}n_j^tB_{j-1}^t
\lesssim_tB_{j-1}^t\leq r^{-t},$$
because $2^{-j}(J+j)^t$ is bounded in $j$. For the lower mass bound, every $x\in A_t$ lies in a retained level-$j$ interval, which has length $B_j^{-1}\leq r$ and carries mass $T_j^{-1}$, so
$$\rho_t(B(x,2r))\geq T_j^{-1}
\geq3^jB_j^{-t}
=3^jn_j^{-t}B_{j-1}^{-t}
\gtrsim_tB_{j-1}^{-t}\geq r^t,$$
and replacing $r$ by $\frac{r}{2}$ gives the asserted bound $\rho_t(B(x,r))\gtrsim_tr^t$.

Finally we bound the mass of small balls from above in order to compute the Hausdorff dimension. An interval of length $2r<2B_{j-1}^{-1}$ meets at most four of the level-$(j-1)$ grid intervals, each of which carries mass at most $T_{j-1}^{-1}$, so $\rho_t(B(x,r))\leq4T_{j-1}^{-1}$ for every $x\in\mathbb{R}$. For $0<u<t$, the inequality \eqref{eq:digit-count} together with $\frac{\log B_j}{\log B_{j-1}}\to1$ gives $T_{j-1}\geq3^{-(j-1)}B_{j-1}^t\gtrsim_uB_j^u$, hence
$$\rho_t(B(x,r))\lesssim_uB_j^{-u}\leq r^u.$$
The mass distribution principle, which says that a probability measure $\rho$ on a set $A$ with $\rho(B(x,r))\lesssim r^u$ for all $x$ and all small $r$ forces $\dim_{\mathcal H}A\geq u$, gives $\dim_{\mathcal H}A_t\geq u$, and letting $u\to t$ shows that $\dim_{\mathcal H}A_t\geq t$. The reverse inequality follows from the covering estimate, since $N(A_t,r)\lesssim r^{-t}$ bounds the upper box dimension of $A_t$, and with it the Hausdorff dimension, by $t$.
\end{proof}

The increasing bases serve two purposes. They make the logarithmic losses in the Fourier estimate smaller than every power of the scale, and they allow the factor $2^{-j}$ in the digit count without changing the dimension. This factor is what gives the covering estimate with exponent exactly $t$ rather than $t+\varepsilon$.

\subsection{The Salem measure}

A compactly supported probability measure on $\mathbb{R}$ whose support has Hausdorff dimension $s$ is called a Salem measure if its Fourier transform decays like $|\xi|^{-\frac{s-\varepsilon}{2}}$ for every $\varepsilon>0$. The exponent $\frac{s}{2}$ cannot be improved, as the Fourier expression for the energy integral used in the proof of Theorem~\ref{Thm:FractalSharpness} below shows, so a Salem measure has as much Fourier decay as its dimension permits. Salem's construction \cite{Salem51} showed that small supports can carry measures of this kind. We use the regular version supplied by Shmerkin \cite[Theorem~A]{Shmerkin}: for each $0<s<1$ there is a probability measure $\sigma$ on $\mathbb{R}/\mathbb{Z}$ which is $s$-Ahlfors--David regular, meaning that $\sigma(B(x,r))$ is bounded above and below by positive constants times $r^s$ for $x\in\operatorname{supp}\sigma$ and $0<r<1$, and whose Fourier coefficients satisfy
$$|\widehat\sigma(k)|\lesssim_\varepsilon
(1+|k|)^{-\frac{s-\varepsilon}{2}}\text{ for }k\in\mathbb{Z}\text{ and }0<\varepsilon<s.$$

A compact set of Hausdorff dimension less than one has Lebesgue measure zero, so its complement contains an open arc. After a translation on the circle we may therefore assume that $\operatorname{supp}\sigma$ avoids a neighborhood of the point $0$, and we lift $\sigma$ to a measure $\lambda$ on $\mathbb{R}$ supported on a compact set $K\subset(0,1)$, keeping the letter $\sigma$ for the translated measure on the circle. Choose $\chi\in C_c^\infty((0,1))$ equal to one on $K$. For fixed $\xi$, the function $x\mapsto\chi(x)e^{-2\pi i\xi x}$ is smooth and vanishes near the endpoints of $[0,1]$, so it extends to a smooth $1$-periodic function whose Fourier series $\sum_k\widehat\chi(\xi+k)e^{2\pi ikx}$ converges uniformly. Integrating against $\sigma$ term by term and replacing $k$ by $-k$, we obtain
$$\widehat\lambda(\xi)
=\int_0^1\chi(x)e^{-2\pi i\xi x}d\lambda(x)
=\sum_{k\in\mathbb{Z}}\widehat\chi(\xi-k)\widehat\sigma(k).$$
The terms with $|\xi-k|\leq\frac{|\xi|}{2}$ have $|k|\geq\frac{|\xi|}{2}$, so their sum is bounded by the supremum of $|\widehat\sigma(k)|$ over this range times $\sum_k|\widehat\chi(\xi-k)|\lesssim1$, while the remaining terms are controlled by the rapid decay of $\widehat\chi$, which makes their sum $O_N((1+|\xi|)^{-N})$ for every $N$. Hence
\begin{equation}\label{eq:Salem-decay}
|\widehat\lambda(\xi)|\lesssim_\varepsilon
(1+|\xi|)^{-\frac{s-\varepsilon}{2}}\text{ for }\xi\in\mathbb{R}.
\end{equation}
The lift remains $s$-Ahlfors--David regular, since balls of radius smaller than the distance from $K$ to $\{0,1\}$ look the same on the line and on the circle. Regularity implies the covering estimate
\begin{equation}\label{eq:Salem-cover}
N(K,r)\lesssim r^{-s},
\end{equation}
because the balls of radius $\frac{r}{2}$ centered at the points of a maximal $r$-separated subset of $K$ are disjoint and each carries mass at least a constant times $r^s$, so there are at most a constant times $r^{-s}$ of them, while the balls of radius $r$ with the same centers cover $K$.

In Shmerkin's construction, the retained patterns are translated independently in different parent intervals. These translations produce Fourier cancellation while preserving the number of retained intervals, and Figure~\ref{fig:structured-salem} compares this with the consecutive digits used in Lemma~\ref{lem:arithmetic}.

\begin{figure}[htbp]
\centering
\resizebox{\textwidth}{!}{
\begin{tikzpicture}[font=\small]

\node[anchor=east] at (2.05,0)     {Initial interval};
\node[anchor=east] at (2.05,-1.1)  {Level 1};
\node[anchor=east] at (2.05,-2.65) {Level 2};

\begin{scope}[xshift=2.4cm]
\node[text=blue!65!black,font=\small] at (2.9,.65) {Structured component};
\path[fill=blue!65!black] (0,-.12) rectangle (5.8,.12);
\draw[blue!65!black!30,densely dotted] (0,-.18) -- (0,-.92);
\draw[blue!65!black!30,densely dotted] (5.8,-.18) -- (5.8,-.92);
\foreach \j in {0,...,11}{
  \path[fill=gray!8,draw=gray!35,line width=.25pt]
    ({5.8*(\j+.07)/12},-1.22) rectangle ({5.8*(\j+.93)/12},-.98);
}
\foreach \j in {0,1,2}{
  \path[fill=blue!65!black,draw=blue!65!black!80!black,line width=.35pt]
    ({5.8*(\j+.07)/12},-1.22) rectangle ({5.8*(\j+.93)/12},-.98);
}
\foreach \parent/\x in {0/0,1/2,2/4}{
  \draw[blue!65!black!35,densely dotted]
    ({5.8*(\parent+.07)/12},-1.28) -- (\x,-2.46);
  \draw[blue!65!black!35,densely dotted]
    ({5.8*(\parent+.93)/12},-1.28) -- ({\x+1.8},-2.46);
  \foreach \j in {0,...,11}{
    \path[fill=gray!8,draw=gray!35,line width=.25pt]
      ({\x+1.8*(\j+.07)/12},-2.77) rectangle ({\x+1.8*(\j+.93)/12},-2.53);
  }
  \foreach \j in {0,1,2}{
    \path[fill=blue!65!black,draw=blue!65!black!80!black,line width=.35pt]
      ({\x+1.8*(\j+.07)/12},-2.77) rectangle ({\x+1.8*(\j+.93)/12},-2.53);
  }
}
\node[text=blue!65!black] at (2.9,-3.2) {Consecutive digits};
\end{scope}

\begin{scope}[xshift=9.1cm]
\node[text=red!70!orange!85!black,font=\small] at (2.9,.65) {Salem component};
\path[fill=red!70!orange] (0,-.12) rectangle (5.8,.12);
\draw[red!70!orange!30,densely dotted] (0,-.18) -- (0,-.92);
\draw[red!70!orange!30,densely dotted] (5.8,-.18) -- (5.8,-.92);
\foreach \j in {0,...,11}{
  \path[fill=gray!8,draw=gray!35,line width=.25pt]
    ({5.8*(\j+.07)/12},-1.22) rectangle ({5.8*(\j+.93)/12},-.98);
}
\foreach \j in {3,4,6}{
  \path[fill=red!70!orange,draw=red!70!orange!80!black,line width=.35pt]
    ({5.8*(\j+.07)/12},-1.22) rectangle ({5.8*(\j+.93)/12},-.98);
}
\foreach \parent/\x/\p/\q/\s in {3/0/0/1/3,4/2/3/4/6,6/4/7/8/10}{
  \draw[red!70!orange!35,densely dotted]
    ({5.8*(\parent+.07)/12},-1.28) -- (\x,-2.46);
  \draw[red!70!orange!35,densely dotted]
    ({5.8*(\parent+.93)/12},-1.28) -- ({\x+1.8},-2.46);
  \foreach \j in {0,...,11}{
    \path[fill=gray!8,draw=gray!35,line width=.25pt]
      ({\x+1.8*(\j+.07)/12},-2.77) rectangle ({\x+1.8*(\j+.93)/12},-2.53);
  }
  \foreach \j in {\p,\q,\s}{
    \path[fill=red!70!orange,draw=red!70!orange!80!black,line width=.35pt]
      ({\x+1.8*(\j+.07)/12},-2.77) rectangle ({\x+1.8*(\j+.93)/12},-2.53);
  }
}
\node[text=red!70!orange!85!black] at (2.9,-3.2) {Independent shifts};
\end{scope}
\end{tikzpicture}
}
\caption{Schematic comparison of the structured and Salem components. Colored cells are retained and pale gray cells are discarded, and each surviving parent is enlarged at Level 2. Consecutive digits provide arithmetic structure, while independent shifts of the retained patterns produce Fourier cancellation.}
\label{fig:structured-salem}
\end{figure}

\subsection{Proof of Theorem~\ref{Thm:FractalSharpness}}

\begin{proof}
Set
$$t=\frac{2\kappa}{n}\text{ and }s=\frac{\alpha-2\kappa}{n}.$$
Then $s>0$ because $\kappa<\frac{\alpha}{2}$, $t\geq0$, and $s+t=\frac{\alpha}{n}<1$, so both parameters lie in the ranges required by Lemma~\ref{lem:arithmetic} and by Shmerkin's theorem. Let $\lambda$ be the Salem measure just described, with support $K$, and let $\rho_t$ be the measure of Lemma~\ref{lem:arithmetic}. Define
$$\nu=\lambda*\rho_t\text{ and }\mu=\nu^{\otimes n},$$
where $\nu^{\otimes n}$ is the product measure on $\mathbb{R}^n$ with $n$ factors equal to $\nu$, so that $\widehat\mu(\xi)=\prod_{i=1}^n\widehat\nu(\xi_i)$ for $\xi=(\xi_1,\ldots,\xi_n)$. Both are probability measures. The support of $\nu$ is the compact set $K+A_t$, and covering each summand at scale $\frac{r}{2}$ covers the sum at scale $r$, so by \eqref{eq:Salem-cover} and Lemma~\ref{lem:arithmetic},
$$N(K+A_t,r)
\leq N\left(K,\frac{r}{2}\right)N\left(A_t,\frac{r}{2}\right)
\lesssim r^{-(s+t)}.$$
Covering each coordinate factor at scale $\frac{r}{\sqrt n}$ then gives
\begin{equation}\label{eq:product-cover}
N(\operatorname{supp}\mu,r)\lesssim r^{-n(s+t)}=r^{-\alpha}.
\end{equation}

We next estimate $\widehat\nu=\widehat\lambda\widehat\rho_t$ on dyadic annuli. For $p\geq1$ we have $|\widehat\rho_t|^p\leq|\widehat\rho_t|$, since $|\widehat\rho_t|\leq1$, and the pointwise bound \eqref{eq:Salem-decay} on $|\widehat\lambda|^p$ combines with the average bound \eqref{eq:arithmetic-L1} on $|\widehat\rho_t|$ to give, after a slight reduction of the losses in each estimate,
\begin{equation}\label{eq:annular-product}
\int_{R\leq|\xi|\leq2R}|\widehat\nu(\xi)|^pd\xi
\lesssim_{p,\varepsilon}
R^{1-t-\frac{ps}{2}+\varepsilon}\text{ for }R\geq1.
\end{equation}
Summing over dyadic values of $R$ and using the product structure, we see that
$$\|\widehat\mu\|_{L^p(\mathbb{R}^n)}^p
=\left(\int_{\mathbb{R}}|\widehat\nu(\xi)|^pd\xi\right)^n<\infty$$
whenever $1-t-\frac{ps}{2}<0$, which is exactly the condition
$$p>p_c=\frac{2(1-t)}{s}
=\frac{2(n-2\kappa)}{\alpha-2\kappa}.$$

The same calculation with $p=2$ gives the dimension assertion. Let $B_R=B(0,R)$. Since $1-s-t>0$, the sum of \eqref{eq:annular-product} with $p=2$ over the dyadic annuli inside $[-R,R]$ is dominated by the outermost one, and taking the $n$-fold product gives
\begin{equation}\label{eq:product-L2upper}
\int_{B_R}|\widehat\mu(\xi)|^2d\xi
\lesssim_\varepsilon R^{n-\alpha+\varepsilon}.
\end{equation}
Recall that for $0<\beta<n$ the energy integral of $\mu$ has the Fourier representation, with a positive constant $c_{n,\beta}$,
$$I_\beta(\mu)
=\iint|x-y|^{-\beta}d\mu(x)d\mu(y)
=c_{n,\beta}\int_{\mathbb{R}^n}
|\widehat\mu(\xi)|^2|\xi|^{\beta-n}d\xi,$$
and that $I_\beta(\mu)<\infty$ implies $\dim_{\mathcal H}(\operatorname{supp}\mu)\geq\beta$. Fix $0<\beta<\alpha$ and choose $0<\varepsilon<\alpha-\beta$. The unit ball contributes a bounded amount to the integral on the right, because $|\widehat\mu|\leq1$, and the dyadic shell $2^j\leq|\xi|<2^{j+1}$ contributes at most $2^{j(\beta-n)}$ times the left side of \eqref{eq:product-L2upper} with $R=2^{j+1}$, so
$$I_\beta(\mu)\lesssim1+\sum_{j\geq0}2^{j(\beta-\alpha+\varepsilon)}<\infty.$$
Thus $\dim_{\mathcal H}(\operatorname{supp}\mu)\geq\beta$ for every $\beta<\alpha$, and the reverse inequality follows from \eqref{eq:product-cover}.

It remains to identify the Fourier ratio exponent. For $p=1$ the exponent $1-t-\frac{s}{2}$ in \eqref{eq:annular-product} is again positive, so the sum over the dyadic annuli inside $[-R,R]$ is again dominated by the outermost one, and taking the product gives
\begin{equation}\label{eq:product-L1ball}
\int_{B_R}|\widehat\mu(\xi)|d\xi
\lesssim_\varepsilon
R^{n(1-t-\frac{s}{2})+\varepsilon}
=R^{n-\frac{\alpha}{2}-\kappa+\varepsilon}.
\end{equation}
Put $b=n-\frac{\alpha}{2}-\kappa$. Decomposing $\mathbb{R}^n$ into $B_R$ and the annuli $B_{2^{j+1}R}\setminus B_{2^jR}$, on which $|\widehat\psi(R^{-1}\xi)|\lesssim_N2^{-jN}$ for any fixed $N>b+1$, we obtain from \eqref{eq:product-L1ball}
$$\|\widehat\mu\widehat\psi(R^{-1}\cdot)\|_1
\lesssim\int_{B_R}|\widehat\mu|
+\sum_{j\geq0}2^{-jN}\int_{B_{2^{j+1}R}}|\widehat\mu|
\lesssim_\varepsilon R^{b+\varepsilon}.$$
On the other hand, $g_R=\mu*\psi_{R^{-1}}$ has integral one and is supported in the $R^{-1}$-neighborhood $(\operatorname{supp}\mu)_{R^{-1}}$ of the support, whose measure is at most a constant times $R^{\alpha-n}$ by \eqref{eq:product-cover} and \eqref{eq:tube}. The Cauchy--Schwarz inequality gives
$$1=\int g_R\leq|(\operatorname{supp}\mu)_{R^{-1}}|^{\frac{1}{2}}\|g_R\|_2
\lesssim R^{\frac{\alpha-n}{2}}\|g_R\|_2,$$
and $\|g_R\|_2=\|\widehat\mu\widehat\psi(R^{-1}\cdot)\|_2$ by Plancherel's theorem. Therefore
$$\operatorname{FR}_{\mu,R}(1)
\lesssim_\varepsilon
R^{-\frac{n}{2}}
\frac{R^{n-\frac{\alpha}{2}-\kappa+\varepsilon}}
{R^{\frac{n-\alpha}{2}}}
=R^{-\kappa+\varepsilon}.$$

For the opposite bound, fix $p>p_c$, so that $\widehat\mu\in L^p(\mathbb{R}^n)$ by the first part of the proof. With $H_R=\widehat\mu\widehat\psi(R^{-1}\cdot)$ and $\theta=\frac{p-2}{2(p-1)}$, H\"older's inequality gives
$$\|H_R\|_2\leq\|H_R\|_1^\theta\|H_R\|_p^{1-\theta},$$
and since $\|H_R\|_p\leq\|\widehat\psi\|_\infty\|\widehat\mu\|_p$ is bounded independently of $R$, this yields $\|H_R\|_1\gtrsim_p\|H_R\|_2^{\frac{1}{\theta}}$. Combined with the lower bound $\|H_R\|_2\gtrsim R^{\frac{n-\alpha}{2}}$ just obtained and $\frac{1-\theta}{\theta}=\frac{p}{p-2}$, this gives
$$\operatorname{FR}_{\mu,R}(1)
=R^{-\frac{n}{2}}\frac{\|H_R\|_1}{\|H_R\|_2}
\gtrsim_pR^{-\frac{n}{2}}\|H_R\|_2^{\frac{p}{p-2}}
\gtrsim_pR^{-\frac{n}{2}+\frac{(n-\alpha)p}{2(p-2)}}.$$
A short computation using $p_c-2=\frac{2(n-\alpha)}{\alpha-2\kappa}$ shows that the exponent equals
$$-\kappa-\frac{(n-\alpha)(p-p_c)}{(p-2)(p_c-2)},$$
which is $-\kappa$ at $p=p_c$. Given $\varepsilon>0$, choosing $p>p_c$ sufficiently close to $p_c$ gives the desired lower bound $\operatorname{FR}_{\mu,R}(1)\gtrsim_\varepsilon R^{-\kappa-\varepsilon}$.
\end{proof}

\section{Hypersurfaces with constant curvature rank}\label{sec:hypersurfaces}

A sphere is curved in every tangent direction and a plane is flat in every tangent direction, while cylinders and smooth pieces of cones lie between these extremes. For a smooth hypersurface $M\subset\mathbb{R}^n$, the second fundamental form at a point is the quadratic form on the tangent space that records the curvature of $M$ in each tangent direction, and its rank is the number of directions in which $M$ bends. When this rank is a constant $k$, the hypersurface has $n-1-k$ flat directions at every point, a quantity known as the relative nullity in the terminology of \cite{Guo93}, and we will see that the Fourier ratio exponent is exactly half the number of flat directions, thus capturing complexity through the notion of curvature.

The integrability threshold for this class was determined by Guo \cite{Guo93}, whose theorem applies to arbitrary distributions supported on the hypersurface and includes the endpoint. We compute the ratio for smooth densities and recover their full integrability range from the estimates above. Surface measure on $M$ is denoted by $\mu_M$.

\begin{proposition}\label{prop:hypersurface}
Let $M\subset\mathbb{R}^n$, $n\geq2$, be a smooth embedded hypersurface whose second fundamental form has constant rank $k$, where $0\leq k\leq n-1$. For every nonzero $f\in C_c^\infty(M)$,
$$\operatorname{FR}_{\mu_M,R}(f)\lesssim R^{-\frac{n-1-k}{2}}
\text{ and }
\kappa_{\mu_M}(f)=\frac{n-1-k}{2}.$$
If $k\geq1$, then for $1\leq p<\infty$ the Fourier transform $\widehat{f\mu_M}$ belongs to $L^p(\mathbb{R}^n)$ if and only if
$$p>\frac{2(k+1)}{k}.$$
If $k=0$, then $\widehat{f\mu_M}$ belongs to no finite $L^p$ space.
\end{proposition}

The two parts of the proposition fit together through Theorem~\ref{thm:FR_synthesis}: with $\alpha=n-1$ and $\kappa=\frac{n-1-k}{2}$, the critical exponent $\frac{2(n-2\kappa)}{\alpha-2\kappa}$ is precisely $\frac{2(k+1)}{k}$, and the case $k=0$ is the case of maximal concentration.

\begin{proof}
First suppose that $f$ is supported in a graph patch
$$M_0=\{(x,\phi(x)):x\in U\},$$
where $U\subset\mathbb{R}^{n-1}$ is open and $\phi$ is smooth. In these coordinates the second fundamental form is a positive multiple of the Hessian of $\phi$, so the Hessian has rank $k$ throughout $U$. Write $a(x)=f(x,\phi(x))J_\phi(x)$, where $J_\phi$ is the surface Jacobian, so that
$$F(\xi,\eta)=\widehat{f\mu_M}(\xi,\eta)
=\int e^{-2\pi i(x\cdot\xi+\eta\phi(x))}a(x)dx$$
for $\xi\in\mathbb{R}^{n-1}$ and $\eta\in\mathbb{R}$. For fixed $\eta$, this is the Fourier transform in $x$ of the function $a(x)e^{-2\pi i\eta\phi(x)}$, whose oscillation is confined to the $k$ directions in which $\phi$ bends. M\"uller's estimate \cite[Theorem~4.1]{Muller82}, also stated in \cite[Lemma~2.3]{Guo93}, quantifies this:
\begin{equation}\label{eq:tangential-L1}
\int_{\mathbb{R}^{n-1}}|F(\xi,\eta)|d\xi
\lesssim_a(1+|\eta|)^{\frac{k}{2}}.
\end{equation}
When $k=0$, the graph function is locally affine and the estimate follows directly, since $F(\xi,\eta)$ is then, up to a unimodular factor, $\widehat a$ evaluated at a translate of $\xi$. For general compactly supported $f$, a finite partition of unity reduces the $L^1$ upper bound to these local estimates. Integrating \eqref{eq:tangential-L1} over $|\eta|\leq R$ gives
$$\int_{B_R}|\widehat{f\mu_M}(\zeta)|d\zeta
\lesssim R^{1+\frac{k}{2}},$$
and dyadic summation against the Schwartz function $\widehat\psi(R^{-1}\cdot)$, exactly as in the proof of Theorem~\ref{Thm:FractalSharpness}, yields
\begin{equation}\label{eq:hypersurface-L1}
\|\widehat{f\mu_M}\widehat\psi(R^{-1}\cdot)\|_1
\lesssim R^{1+\frac{k}{2}}.
\end{equation}

The compact support of $f\mu_M$ is a compact subset of a smooth hypersurface and therefore satisfies the covering bound with $\alpha=n-1$, so Lemma~\ref{lem:basic} gives
\begin{equation}\label{eq:hypersurface-L2}
\|\widehat{f\mu_M}\widehat\psi(R^{-1}\cdot)\|_2
\gtrsim R^{\frac{1}{2}}.
\end{equation}
Dividing \eqref{eq:hypersurface-L1} by $R^{\frac{n}{2}}$ times \eqref{eq:hypersurface-L2} proves the ratio upper bound $\operatorname{FR}_{\mu_M,R}(f)\lesssim R^{-\frac{n-1-k}{2}}$, and hence $\kappa_{\mu_M}(f)\geq\frac{n-1-k}{2}$. When $k=0$, the lower bound in Lemma~\ref{lem:basic} gives the same power in the other direction, so $\kappa_{\mu_M}(f)=\frac{n-1}{2}=\frac{\alpha}{2}$, and the maximal concentration case of Theorem~\ref{thm:FR_synthesis} excludes every finite $p$.

Suppose now that $k\geq1$. For a graph patch, Littman's estimate \cite{Litt63} gives the pointwise decay
$$|F(\xi,\eta)|\lesssim_{M,f}
(1+|\xi|+|\eta|)^{-\frac{k}{2}},$$
while Plancherel's theorem in the tangential variable gives, for each fixed $\eta$,
$$\int_{\mathbb{R}^{n-1}}|F(\xi,\eta)|^2d\xi
=\|a\|_2^2.$$
For $p>2$ we write $|F|^p=|F|^{p-2}|F|^2$, bound the first factor by Littman's estimate with $|\xi|$ discarded, and integrate the second in $\xi$, which gives
$$\int_{\mathbb{R}^n}|F(\xi,\eta)|^pd\xi d\eta
\lesssim_{M,f,p}
\int_{\mathbb{R}}(1+|\eta|)^{-\frac{k(p-2)}{2}}
\int_{\mathbb{R}^{n-1}}|F(\xi,\eta)|^2d\xi d\eta
\lesssim_{M,f,p}
\int_{\mathbb{R}}(1+|\eta|)^{-\frac{k(p-2)}{2}}d\eta.$$
The last integral is finite exactly when $\frac{k(p-2)}{2}>1$, that is, when $p>\frac{2(k+1)}{k}$, and a finite partition of unity proves the same assertion for arbitrary $f\in C_c^\infty(M)$.

Fix such a $p$. H\"older's inequality, the bound $\|\widehat{f\mu_M}\widehat\psi(R^{-1}\cdot)\|_p\lesssim\|\widehat{f\mu_M}\|_p$, which is uniform in $R$, and the lower bound \eqref{eq:hypersurface-L2} give, as in the last step of the proof of Theorem~\ref{Thm:FractalSharpness},
$$\operatorname{FR}_{\mu_M,R}(f)
\gtrsim_pR^{-\frac{n}{2}+\frac{p}{2(p-2)}},$$
and consequently
$$\kappa_{\mu_M}(f)
\leq\frac{n}{2}-\frac{p}{2(p-2)}.$$
Letting $p$ decrease to $\frac{2(k+1)}{k}$, for which $\frac{p}{2(p-2)}=\frac{k+1}{2}$, proves $\kappa_{\mu_M}(f)\leq\frac{n-1-k}{2}$, and the ratio upper bound gives the reverse inequality. Finally, Theorem~\ref{thm:FR_synthesis} excludes every $p$ below the critical exponent $\frac{2(k+1)}{k}$, and Proposition~\ref{prop:endpoint}, whose hypothesis is the uniform ratio bound just established, excludes the endpoint.
\end{proof}

\begin{figure}[htbp]
\centering
\resizebox{\linewidth}{!}{
\begin{tikzpicture}[font=\small]
\begin{scope}
\begin{scope}
\clip (0,0) circle[radius=1.6];
\foreach \r/\o in {0/0.55,0.2/0.12,0.4/0.42,0.6/0.10,0.8/0.32,1.0/0.08,1.2/0.24,1.4/0.07}
  {\fill[gray,fill opacity=\o,even odd rule] (0,0) circle[radius=\r] (0,0) circle[radius={\r+0.2}];}
\end{scope}
\draw[dashed] (0,0) circle[radius=1.6];
\draw[->,gray] (-1.9,0) -- (1.9,0) node[right,black] {$\xi$};
\draw[->,gray] (0,-1.9) -- (0,1.9);
\node at (0,2.5) {sphere $S^2$};
\node at (0,-2.3) {$\kappa=0$};
\end{scope}
\begin{scope}[shift={(4.9,0)}]
\begin{scope}
\clip (0,0) circle[radius=1.6];
\foreach \x/\o in {0/0.55,0.2/0.12,0.4/0.42,0.6/0.10,0.8/0.32,1.0/0.08,1.2/0.24,1.4/0.07}
  {\fill[gray,fill opacity=\o] (\x,-0.25) rectangle ({\x+0.2},0.25);
   \fill[gray,fill opacity=\o] ({-\x-0.2},-0.25) rectangle ({-\x},0.25);}
\end{scope}
\draw[dashed] (0,0) circle[radius=1.6];
\draw[->,gray] (-1.9,0) -- (1.9,0) node[right,black] {$\xi$};
\draw[->,gray] (0,-1.9) -- (0,1.9) node[left,black] {$\eta$};
\node at (0,2.5) {cylinder $S^1\times\mathbb{R}$};
\node at (0,-2.3) {$\kappa=\frac{1}{2}$};
\end{scope}
\begin{scope}[shift={(9.8,0)}]
\begin{scope}
\clip (0,0) circle[radius=1.6];
\fill[gray,fill opacity=0.4] (-0.25,-1.6) rectangle (0.25,1.6);
\end{scope}
\draw[dashed] (0,0) circle[radius=1.6];
\draw[->,gray] (-1.9,0) -- (1.9,0) node[right,black] {$\xi$};
\draw[->,gray] (0,-1.9) -- (0,1.9) node[left,black] {$\eta$};
\node at (0,2.5) {plane $\mathbb{R}^2\times\{0\}$};
\node at (0,-2.3) {$\kappa=1$};
\end{scope}
\end{tikzpicture}
}
\caption{Schematic sections through frequency space for smooth compactly supported densities on a sphere, a cylinder and a plane in $\mathbb{R}^3$. The Fourier mass of the sphere fills the ball of radius $R$ in oscillating shells, the cylinder confines it to a slab around the flat direction, and the plane confines it to a line. The indicated Fourier ratio exponents follow from Proposition~\ref{prop:hypersurface}.}
\label{fig:examples}
\end{figure}

\begin{remark}
The sharpness in Theorem~\ref{Thm:FractalSharpness} concerns the information contained in the pair $(\alpha,\kappa)$. It does not assert that this pair determines the optimal synthesis threshold for every individual support, and additional geometry may give stronger conclusions, as the results for curves in \cite{GIZZ23,DPR} illustrate. For the hypersurfaces in Proposition~\ref{prop:hypersurface}, the ratio calculation does recover the threshold.
\end{remark}

\textbf{Funding Support:}

A. I. was supported in part by NSF grants DMS-2154232 and DMS-2506858 and by a Simons Fellowship in Mathematics in 2026.

\end{document}